\documentclass[letterpaper, 10 pt, conference]{ieeeconf}

\IEEEoverridecommandlockouts                              
 \pdfoutput=1
\usepackage{amsmath, amssymb}
\usepackage{amsfonts}
\usepackage{graphicx}
\usepackage{enumerate}
\usepackage{caption}
\usepackage{ifthen}
\usepackage{multicol}
\usepackage{float}
\usepackage{cite}
\usepackage[usenames, dvipsnames]{color}
\usepackage[lofdepth,lotdepth]{subfig}
\graphicspath{{figures/}}
\usepackage{url}

\usepackage{adjustbox}
\usepackage{multirow}

\floatstyle{plaintop}
\restylefloat{table}

\newtheorem{theorem}{Theorem}[section]
\newtheorem{lemma}[theorem]{Lemma}
\newtheorem{proposition}[theorem]{Proposition}
\newtheorem{corollary}[theorem]{Corollary}
\newtheorem{rem}{Remark}
\newtheorem{example}{Example}
\newtheorem{ass}{Assumption}

\newtheorem{excont}{Example}

\newboolean{showcomments}
\setboolean{showcomments}{true}

\newcommand{\henrik}[1]{\ifthenelse{\boolean{showcomments}}
{\textcolor{Blue}{Henrik says: #1}}{}}
\newcommand{\emma}[1]{\ifthenelse{\boolean{showcomments}}
{\textcolor{Red}{Emma says: #1}}{}}
\newcommand{\martin}[1]{\ifthenelse{\boolean{showcomments}}
{\textcolor{Skyblue}{Martin says: #1}}{}}
\newcommand{\todo}[1]{\ifthenelse{\boolean{showcomments}}
{\textcolor{Green}{To Do: #1}}{}}

\newboolean{shownew}
\setboolean{shownew}{true}
\newcommand{\newtext}[1]{\ifthenelse{\boolean{shownew}}
{{#1}}{}}

\newcommand{\hn}{$\mathcal{H}_2$ }

\newcommand{\cl}{\mathcal{L}}
\newcommand{\ldn}{\lambda_n}

\newcommand*\wideestimates{\mathrel{\widehat{=}}}

\usepackage{array}
\newcolumntype{L}[1]{>{\raggedright\let\newline\\\arraybackslash\hspace{0pt}}m{#1}}
\newcolumntype{C}[1]{>{\centering\let\newline\\\arraybackslash\hspace{0pt}}m{#1}}
\newcolumntype{R}[1]{>{\raggedleft\let\newline\\\arraybackslash\hspace{0pt}}m{#1}}
\usepackage{subfig}

 \usepackage{tikz}
 \usetikzlibrary{decorations.pathmorphing,shadings}
 \usetikzlibrary{plotmarks}
 \usepackage{pgfplots}
\usepackage{tkz-graph}

\definecolor{gray3}{rgb}{0.75, 0.75, 0.75}
\definecolor{gray2}{rgb}{0.5, 0.5, 0.5}
\definecolor{gray1}{rgb}{0.25, 0.25, 0.25}
\definecolor{gray0}{rgb}{0.15, 0.15, 0.15}

\pgfplotscreateplotcyclelist{mygraylist}{%
black, densely dashed\\%
gray1\\%
gray3\\%
gray2\\%
gray2,dashed\\%
gray3,dashed\\%
}

\pgfplotscreateplotcyclelist{myothergraylist}{%
black, densely dashed\\%
gray1\\%
gray2\\%
gray3\\%
gray1, dashed\\%
gray3,dashed\\%
}

\usetikzlibrary{arrows}

\renewcommand{\baselinestretch}{0.979}

\title{\LARGE \bf On the Coherence of Large-Scale Networks \\ with Distributed PI and PD Control}

\author{Emma Tegling and Henrik Sandberg %
\thanks{The authors are with the School of Electrical Engineering and the ACCESS Linnaeus Center, KTH Royal Institute of Technology, SE-100 44 Stockholm, Sweden {(\tt tegling, hsan@kth.se)}. }  \thanks{This work was supported in part by the Swedish Research Council through grants 2013-5523 and 2016-00861.}
 }%

\begin{document}
\maketitle
\thispagestyle{empty}
\pagestyle{empty}

\begin{abstract}
We consider distributed control of double-integrator networks, where agents are subject to stochastic disturbances. We study performance of such networks in terms of \emph{coherence}, defined through an \hn norm metric that represents the variance of nodal state fluctuations. Specifically, we address known performance limitations of the standard consensus protocol, which cause this variance to scale unboundedly with network size for a large class of networks. We propose distributed proportional integral (PI) and proportional derivative (PD) controllers that relax these limitations and achieve bounded variance, in cases where agents can access an absolute measurement of one of their states. This case applies to, for example, frequency control of power networks and vehicular formation control with limited sensing. We discuss optimal tuning of the controllers with respect to network coherence and demonstrate our results in simulations. 



 
%
%
\end{abstract}


\section{Introduction}
\label{sec:intro}
The problem of distributed control of networked systems has been extensively studied over the past decades \cite{Jadbabaie2003, OlfatiSaber, RenAtkins2005}. \newtext{A well-studied subproblem in this context is that of consensus,} 
 that is, to drive the network of agents to the same state. When the system is subject to external disturbances or uncertainties, there are, however, limitations to the performance achievable through distributed control. In particular, there are limits on how close the agents of a large-scale network can come to a state of consensus, or what is often called the \textit{coherence} of the network. Such fundamental limitations were studied in~\cite{Bamieh2012}, where scalings of measures of network coherence with network size were established for first- and second-order consensus networks. It was shown that reasonable performance in sparse networks, such as vehicle platoons, requires that each agent can access measurements of its own states with respect to a global reference frame; what we refer to as \emph{absolute feedback}. 


The importance of absolute feedback for network coherence was also recognized in \cite{Patterson2010, Lin2014, Pirani2015}, where the impact of leaders (select agents with access to absolute measurements) was studied in first-order consensus networks. 
In this work, we will focus on problems of second-order consensus, and \newtext{do not consider} adding absolute measurements to improve performance, which would require increasing the agents' sensory capabilities. \newtext{Instead,} we propose the use of alternative controller structures, namely controllers with integral and derivative action in addition to standard proportional control. 

Proportional-integral-derivative (PID) control has been studied for consensus networks in \cite{Lombana2015} and the role of integral action in disturbance rejection in both first- and second-order consensus networks was described in \cite{Freeman2006, Andreasson2014TAC, Seyboth2015}. While these works have focused on proving convergence of the respective control strategies, we address the question of their performance and their scalability to large-scale networks.
In particular, we show that for certain cases of second-order consensus considered in~\cite{Bamieh2012}, distributed PI and PD control can relax the performance limitations found therein. 

We study cases in which the agents have access to absolute measurements of only one of their states and evaluate performance in terms of an \hn norm metric representing network coherence. For these cases, we compare proportional control, with which the \hn norm metric scales unboundedly with network size for a large class of networks, to distributed PI and PD control, where the same metric can be shown to be bounded for any network. The apparent reason for this improved scalability is that integral or derivative action, when applied to the available absolute state measurement, emulates absolute feedback from the other state. For example, the derivative of a position measurement corresponds to a velocity measurement. While this may seem intuitive for \emph{ideal} integral or derivative action, the controllers we consider here are modified to account for imperfections. In particular, we first consider a \emph{distributed averaging PI (DAPI)} controller, in which the integral state is passed through a consensus filter to prevent destabilizing drift due to measurement noise (``dead-reckoning''). 
\newtext{ A version of this controller has previously been studied in the context of electric power networks in~\cite{Andreasson2014ACC, SimpsonPorco2015, Tegling2016ACC}. }
Second, we model low-pass filtering of the derivative action in the proposed \emph{filtered distributed PD (F-DPD)} controller. Interestingly, the improved scalability is achieved with any design of these filters, yet this paper also provides insights into how their design impacts performance, as well as to \newtext{their optimal tunings}.

The remainder of this paper is organized as follows. We introduce the problem setup and the performance metric in Section~\ref{sec:prels}, where we also give examples from applications. In Section~\ref{sec:limits}, we discuss known limitations of proportional control. We introduce the DAPI and F-DPD controllers and evaluate their performance in Section~\ref{sec:PID}, then discuss their tuning in Section~\ref{sec:tuning}. We conclude in Section~\ref{sec:discussion}.

\section{Preliminaries}
\label{sec:prels}
\subsection{Notation}
Consider a network of $N$ identical agents modeled by the undirected graph $\mathcal{G} = \{\mathcal{V}, \mathcal{E}\}$, where $\mathcal{V} = \{1,2,\ldots,N\}$ is the set of nodes and $\mathcal{E} \subset \mathcal{V} \times \mathcal{V}$ is the set of edges. Denote by $\mathcal{N}_i$ the neighbor set of agent $i$ in $\mathcal{G}$. Assume that each edge in $\mathcal{E}$ has an associated weight $l_{ij}  = l_{ji}{>0}$ and denote by $\mathcal{L}$ the weighted graph Laplacian whose elements $\mathcal{L}_{ij} = \sum_{k = 1,~k\neq i}^N l_{ik}$ if $i = j$ and $-l_{ij}$ otherwise.  Throughout this paper, the graph is assumed to be connected. The eigenvalues of $\cl$ can thus be written $0 = \lambda_1 < \lambda_2 \le \ldots \le \lambda_N$.  

\subsection{Problem set-up}
We consider a set $\mathcal{V}$ of agents, each governed by double-integrator dynamics and subject to stochastic disturbances:
\begin{equation}
\label{eq:doubledyn}
\begin{aligned}
\dot{x}_i(t) & = v_i(t) \\
\dot{v}_i(t) & = u_i(t) + w_i(t).
\end{aligned}
\end{equation}
Here, $u_i(t)$ is a control input and $w_i(t)$ is assumed to be a zero-mean Gaussian white noise process that is uncorrelated across nodes. Henceforth, we will often drop the time-dependence in the notation. Without loss of generality, we assume that the states $x_i,v_i \in \mathbb{R}$ for each agent $i$ represent deviations from a desired trajectory $\bar{x}_i$ with common constant velocity $\bar{v}$, so that $\bar{x}_i(t) := \bar{v}t +\delta_i$, where $\delta_i$ is a given setpoint.  The states $x_i,v_i$ may carry different meanings depending on the application. We provide two examples at the end of this section. 


The control objective is for all agents to follow the desired trajectory, in our case, to drive the states to zero. 
\newtext{We model full state feedback control} and first consider a standard linear consensus algorithm (see, for example,~\cite{RenAtkins2005})
\begin{equation}
\label{eq:static}
u_i = -\!\! \sum_{j \in \mathcal{N}_i}\! f_{ij}(x_i - x_j)  -\!\! \sum_{j \in \mathcal{N}_i} \! g_{ij}(v_i - v_j) -f_0x_i - g_0v_i,
\end{equation}
where $f_{ij},~ g_{ij},~f_0,~g_0$ are nonnegative gains. We refer to \eqref{eq:static} as \emph{proportional (P) control}, since the feedback is proportional to state measurements at time $t$.

Throughout this paper, we distinguish between two types of state measurements and feedback: \emph{relative} and \emph{absolute}. Relative measurements are taken with respect to neighbors, as in $x_i - x_j,~v_i - v_j$ for $j \in \mathcal{N}_i$, while absolute measurements imply that agent $i$ can access its own state $x_i$ or $v_i$. In \eqref{eq:static}, we say that \textit{absolute feedback} from the state $x_i$ ($v_i$) exists if $f_0>0$ ($g_0>0$). In the following, we will consider cases in which absolute measurements are only available for \textit{either} $x_i$ or $v_i$, \newtext{and let this apply to all agents. Regarding the assumption on uniform gains $f_0,g_0$, see Remark~\ref{rem:relaxf0g0}.}

By defining the state vectors $x = [x_1,\ldots,x_N]^T$ and $v = [v_1,\ldots,v_N]^T$, we can write the system \eqref{eq:doubledyn} with control \eqref{eq:static} as
\begin{equation}
\label{eq:staticvec}
\begin{bmatrix}
\dot{x} \\ 
\dot{v}
\end{bmatrix} = \begin{bmatrix}
0 & I \\ 
-\mathcal{L}_F - f_0I & -\mathcal{L}_G -g_0I
\end{bmatrix}\begin{bmatrix}
x\\ v
\end{bmatrix} + \begin{bmatrix}
0 \\ I
\end{bmatrix} w,
\end{equation}
where $\mathcal{L}_{F(G)}$ is the weighted graph Laplacian with edge weights $f_{ij}~(g_{ij})$ for $\{i,j\} \in \mathcal{E}$, and $w$ is the noise vector. 

In order to provide tractable closed-form solutions in what follows, we will impose the following assumption:
\begin{ass}[Proportional gains] 
\label{ass:proportional}
The ratio $f_{ij}/g_{ij}$ is uniform across all $\{i,j\} \in \mathcal{E}$. We write $f_{ij} = fl_{ij}$ and $g_{ij} = gl_{ij}$, so that $\mathcal{L}_F = f\mathcal{L}$ and $\mathcal{L}_G = g\mathcal{L}$, \newtext{with $f,g\ge0$}.
\end{ass}

\begin{example}[Vehicular formation]
Consider a set of $N$ vehicles in a formation, where the control objective for each vehicle is to follow the trajectory $\bar{x}_i(t) := \bar{v}t +i\Delta $, despite being subject to random forcings $w_i$. Here, $\bar{v}$ is a common cruising velocity and $\Delta$ is the desired inter-vehicle spacing. 
Each vehicle controls its velocity according to \eqref{eq:static}, with nearest-neighbor interactions, so that $\mathcal{N}_i = \{i +1,i-1\}$ for $i = 2,\ldots,{N-1}$, $\mathcal{N}_1 = \{2\}$ and $\mathcal{N}_N = \{N-1\}$, resulting in a 1-D string formation. The closed-loop system becomes~\eqref{eq:staticvec}.

Consider a scenario in which the vehicles are not equipped with speedometers, but have radars to measure relative positions and velocities with respect to neighbors. Furthermore, the position of the lead vehicle is broadcast across the network, so each vehicle can calculate its own position. 
\newtext{In this example,} there is therefore absolute feedback from the position $x$, but not from the velocity $v$, so $f_0>0$ while $g_0 = 0$ in~\eqref{eq:static}. 
\end{example}

\begin{example}[Frequency control in power network]
\label{ex:power}
Synchronization in power networks is typically studied through a system of coupled swing equations. Under some simplifying assumptions, the linearized swing equation 
can be written as: 
\begin{equation}
\label{eq:swingeq}
m \ddot{\theta}_i + d\dot{\theta}_i = -\sum_{j \in \mathcal{N}_i} b_{ij}(\theta_i - \theta_{j}) +P_{m,i}, 
\end{equation}
where $\theta_i$ is the phase angle deviation at node $i$, $\dot{\theta}_i = \omega_i$ is the frequency deviation,  and $m_i$ and $d_i$ are, respectively, inertia and damping coefficients. The term $P_{m,i}$ can be seen as the net power injection at the node, and $b_{ij}$ is the susceptance of the $(i,j)^{\mathrm{th}}$ power line. In the absence of any additional control input, we also refer to \eqref{eq:swingeq} as \textit{frequency droop control}.

The system~\eqref{eq:swingeq} can be cast as the P-controlled system~\eqref{eq:staticvec}, with $x \wideestimates \theta$, $v \wideestimates \omega $, $g_0 \wideestimates \frac{d}{m}$, $\mathcal{L}_F \wideestimates \frac{b_{ij}}{l_{ij} m}\mathcal{L}$, $\cl_G = 0$ and $f_0 = 0$. The power injection $P_{m}$ can be absorbed into the disturbance input $w$, which we take to represent random fluctuations in generation and load. In this problem, there is absolute feedback from the frequency $\omega$, but only relative feedback from the phase angles $\theta$ (absolute measurement of phase angles would require phasor measurement units (PMUs), which in general are not available).

\subsection{Performance metric}
We evaluate performance of control laws in terms of measures of network \emph{coherence}. Coherence can be understood as a measure of network disorder, or in other words, how well the control objective of consensus is achieved. Similar to \cite{Bamieh2012, Patterson2010}, we define it as the steady-state variance of the agents' deviation from the network average:
\begin{equation}
\label{eq:perfmeas}
V = \lim_{t \rightarrow \infty} \sum_{i \in \mathcal{V}} \mathbb{E}\left\lbrace \left( x_i(t) - \frac{1}{N}\sum_{j \in \mathcal{V}} x_j(t) \right)^2 \right\rbrace.
\end{equation}
The quantity $V$ \newtext{can be evaluated analytically as} the squared \hn norm of an input-output system $\mathcal{S}$ from the disturbance input~$w(t)$ in \eqref{eq:doubledyn} to a performance output $y(t)$ defined as
\begin{equation}
\label{eq:ydef}
y(t)  = \left(I - \frac{1}{N}\mathbf{1}\mathbf{1}^T \right)x(t),
\end{equation}
where $\mathbf{1}$ is the ($N\times 1$) vector of all ones. 

To better analyze the scaling of \eqref{eq:perfmeas} with the network size, we will normalize it by the total number of agents:
\begin{equation}
\label{eq:pernodevariance}
V_N = \frac{1}{N}V.
\end{equation}
It is the scaling of $V_N$ as $N$ increases that we can refer to as the level of coherence in the network. 
For a given control law to be \emph{scalable}, $V_N$ should be uniformly bounded, in which case the system can also be regarded as fully coherent. 


\end{example}

\section{Limitations of proportional control}
\label{sec:limits}
We begin the discussion on limitations of the proportional control law~\eqref{eq:static} by stating the closed-form expression for its performance.

\begin{lemma} \label{thm:static}
The scaled \newtext{performance }output variance \eqref{eq:pernodevariance} for the P-controlled system \eqref{eq:staticvec} is given by
\begin{equation}
\label{eq:normstatic}
V_N^{\mathrm{P}} = \frac{1}{2N}\sum_{n = 2}^N \frac{1}{ (f_0 + f \lambda_n) (g_0 + g \lambda_n)}.
\end{equation}
\end{lemma}
 \vspace{1mm}
\begin{proof}
See Appendix.
\end{proof}
Clearly, in the absence of absolute feedback (that is, if $f_0 = 0$ and/or $g_0 = 0$), the sum in \eqref{eq:normstatic} tends towards infinity if one or more of the Laplacian eigenvalues $\lambda_n$, $n\ge 2$ approaches zero. Laplacian eigenvalues typically tend to zero as a network grows large, unless it is densely interconnected. If the sum increases faster than linearly in $N$, the performance metric~$V_N^P$ will scale badly and the network lacks coherence.  
This was the problem considered in \cite{Bamieh2012}, where asymptotic scalings of expressions like \eqref{eq:normstatic} were derived for networks built over $d$-dimensional torical lattices~$\mathbb{Z}_N^d$. Their results for the most problematic 1-dimensional network (a ring graph) \newtext{show that $V_N^P$ scales linearly in $N$ if one of $f_0,g_0$ is zero, and as $N^3$ if both are zero.   }
The same scalings can be derived using arguments based on effective resistances for networks that can be embedded in lattice networks, see~\cite{Barooah}.

\newtext{This points} at a fundamental limitation to the performance of the P controller~\eqref{eq:staticvec}, as the variance~\eqref{eq:normstatic} is only bounded for any network if there is absolute feedback from both the states $x_i$ and $v_i$. That is, both measurements are required for P control to be scalable. 
An important question 
has therefore been whether alternative controller structures can alleviate these limitations. In the next section, we present linear controllers for which it suffices to have an absolute measurement of either $x_i$ or $v_i$ to achieve bounded scalings for any network.

\begin{rem}
\label{rem:relaxf0g0}
\newtext{The assumption that the absolute feedback gains $f_0,g_0$ are uniform across the network is needed to derive the closed-form expression~\eqref{eq:normstatic}, but can be shown not to be important for the main conclusion that $V_N^P$ scales badly in the absence of absolute feedback. With non-uniform gains $f_0,g_0$, bounds on $V_N^P$ on the form~\eqref{eq:normstatic} can instead be stated in terms of the minimum and maximum available gains. }
\end{rem}

\section{Distributed PI and PD control }
\label{sec:PID}
We now introduce the distributed PI and PD controllers with which we propose to address the limitations posed by proportional control. We show that these controllers achieve bounded output variances~\eqref{eq:pernodevariance}, and demonstrate the performance improvement through numerical examples. 
\subsection{Absolute $v$-feedback: Distributed PI control}
\label{sec:dapi}
Suppose that an absolute measurement of $v_i$ is available, while one of $x_i$ is not. 
For this case, we propose to use the following \emph{distributed \newtext{averaging} proportional-integral \newtext{(DAPI)}} controller:
\begin{align} \nonumber
u_i &= - \!\! \sum_{j \in \mathcal{N}_i}\!  f_{ij}(x_i - x_j)  -\! \!\sum_{j \in \mathcal{N}_i} \! g_{ij}(v_i - v_j)  - g_0v_i + K_Iz_i\\
\dot{z}_i &= -v_i - \!\! \sum_{j \in \mathcal{N}_i}c_{ij}(z_i - z_j),
\label{eq:dapi}
\end{align}
where $z_i$ is the integral state, $K_I$ is a positive gain (integral gain), and we will discuss the averaging filter $-\sum_{j \in \mathcal{N}_i}c_{ij}(z_i - z_j)$ shortly. First, note that if $c_{ij} = 0$ for all $\{i,j\}$, it ideally holds that $z_i(t) = \int_0^t  v_i(\tau) d\tau + z_i(0)$ and since, by definition, $\int_0^t  v_i(\tau) d\tau = x_i(t) - x_i(0)$, the integral state $z_i$ would correspond to an absolute measurement of the state $x_i$, modulo initial values. In this case, we would therefore expect that the controller~\eqref{eq:dapi} has a performance similar to \eqref{eq:static} with absolute feedback from both $v_i$ and $x_i$. 
\newtext{However, as mentioned in the introduction, an averaging filter on the integral states with positive gains $c_{ij}$ is required to ensure system stability~\cite{Andreasson2014TAC}.} 
We will show that also with this filter, the desired performance scaling can be achieved.

By inserting \eqref{eq:dapi} into \eqref{eq:doubledyn} we obtain, in vector form:
\begin{equation}
\label{eq:DAPIvec}
\begin{bmatrix}
\dot{x} \\ 
\dot{v} \\ \dot{z}
\end{bmatrix} = \begin{bmatrix}
0 & I & 0 \\ 
-\mathcal{L}_F & -\mathcal{L}_G -g_0I & K_I \\
0 & -I & -\mathcal{L}_C
\end{bmatrix}\begin{bmatrix}
x\\ v \\ z
\end{bmatrix} + \begin{bmatrix}
0 \\ I \\ 0
\end{bmatrix} w.
\end{equation}
In line with Assumption~\ref{ass:proportional}, we make the following assumption on the weighted Laplacian matrix $\mathcal{L}_C$:
\begin{ass}
\label{ass:uniformC}
The gains $c_{ij}$ are proportional to $f_{ij},g_{ij}$ for all $\{i,j\}\in \mathcal{E}$. We write $c_{ij} = cl_{ij}$, so that $\mathcal{L}_C = c\mathcal{L}$, \newtext{with $c\ge 0$}.
\end{ass}\vspace{0.5mm}
Under Assumptions~\ref{ass:proportional}--\ref{ass:uniformC}, \newtext{the system~\eqref{eq:DAPIvec} is input-output stable with respect to the output~\eqref{eq:ydef}~\cite{Tegling2016Lic} and} its performance can be stated as follows:
\vspace{0.5mm}
\begin{proposition}
\label{thm:dapiresult}
The scaled \newtext{performance} output variance~\eqref{eq:pernodevariance} for the DAPI-controlled system~\eqref{eq:DAPIvec} is given by
\begin{small}
\begin{equation}
\label{eq:normdapi}
V_N^{\mathrm{DAPI}} \!\! = \frac{1}{2N} \!\! \sum_{n = 2}^N \frac{1}{ fg\lambda_n^2+ \frac{K_If(g_0 + \lambda_n(c+g))+g_0f\lambda_n(c^2\lambda_n + f + cg_0 )}{f + c g_0 + c \lambda_n (c+g)} }.
\end{equation}
\end{small}
\end{proposition}
 \vspace{2mm}
\begin{proof}
The result~\eqref{eq:normdapi} is derived in the same manner as in the proof of Lemma~\ref{thm:static}, and the details are omitted for the sake of brevity.
\end{proof}

While the expression~\eqref{eq:normdapi} is convoluted, it allows for the following important conclusion:
\begin{corollary}
\label{thm:dapicor}
For any positive and finite gains $K_I$ and $c$, $V_N^{\mathrm{DAPI}}$ in~\eqref{eq:normdapi} is uniformly bounded in $N$. It holds that 
\begin{equation}
\label{eq:boundsdapi}
0 < V_N^{\mathrm{DAPI}} < \frac{f+cg_0}{2K_Ifg_0}.
\end{equation}
\end{corollary}\vspace{2mm}
\begin{proof}
It is readily verified that the expression $\frac{1}{ fg\lambda_n^2+ \frac{K_If(g_0 + \lambda_n(c+g))+g_0f\lambda_n(c^2\lambda_n + f + cg_0) }{f + c g_0 + c \lambda_n (c+g)} } :=s_n$ from the sum in~\eqref{eq:normdapi} is monotonically decreasing in $\lambda_n \ge 0$. Its supremum $\bar{s}_n$ is therefore obtained as $\lambda_n \rightarrow 0$ and is $\bar{s}_n = \frac{f+cg_0}{K_Ifg_0}$. Thus, the sum $\sum_{n = 2}^N s_n< (N-1)\bar{s}_n < N\frac{f+cg_0}{K_Ifg_0}$, which, inserted in \eqref{eq:normdapi}, gives the upper bound in \eqref{eq:boundsdapi}. The lower bound is obtained when $\lambda_n \rightarrow \infty$ for all $n = 2,\ldots, N$. This is, e.g., the case when $\mathcal{G}$ is the complete graph $\mathcal{K}_N$, as $N \rightarrow \infty$.         \end{proof}
\begin{rem}
\label{rem:zeroc}
If $c\rightarrow 0$ in \eqref{eq:normdapi}, the output variance for P control with absolute feedback from both $x_i$ and $v_i$ is retrieved, substituting the integral gain $K_I$ for $f_0$ in~\eqref{eq:normstatic}. This is in line with the discussion at the beginning of this section. Interestingly though, letting $c \rightarrow 0$ (which is only possible in theory) does not necessarily minimize $V_N^\mathrm{DAPI}$. We will discuss how to optimize performance for $c$ in Section~\ref{sec:tuning}.
\end{rem}
\begin{excont}[Continued]
Applying DAPI control to the power system dynamics~\eqref{eq:swingeq} yields the closed-loop system:
 \begin{equation}
\label{eq:dapiexample}
\begin{aligned}
 \dot{\omega}_i &= - \frac{d}{m}\theta_i -\frac{1}{m}\sum_{j \in \mathcal{N}_i} b_{ij}(\theta_i - \theta_j) + K_I z_i \\  
 \dot{z}_i & = -\omega_i - \sum_{j \in \mathcal{N}_i}c_{ij}(z_i - z_j)
\end{aligned}
\end{equation}
In this case, the integral action 
is also referred to as \emph{secondary frequency control}, the role of which is to eliminate any stationary frequency control errors that arise with just P control \cite{Andreasson2014ACC, SimpsonPorco2015}. As we have shown in this section, it also improves transient performance.

In Fig.~\ref{fig:dapisim}, we show a simulation of P control~\eqref{eq:swingeq} and DAPI control~\eqref{eq:dapiexample} on a radial power network \newtext{(here modeled by a path graph)} with, respectively, 10 and 100 nodes. The figures show that P control does not scale well to the larger network, while DAPI control does.


\end{excont}

\begin{figure*}
  \centering
  \subfloat[][Phase trajectories in power network with path graph topology in which $N = 10$ and $N=100$, under random initial frequency perturbations (only subset shown for $N=100$).  
  Here, $m = \frac{20}{\omega^{\mathrm{ref}}}$, $d = \frac{10}{\omega^{\mathrm{ref}}}$, with $\omega^{\mathrm{ref}} = 2\pi 60$ Hz, $K_I = 1$, $c = 0.1$, $b_{ij} = 0.3$.]{
 \includegraphics[scale = 1]{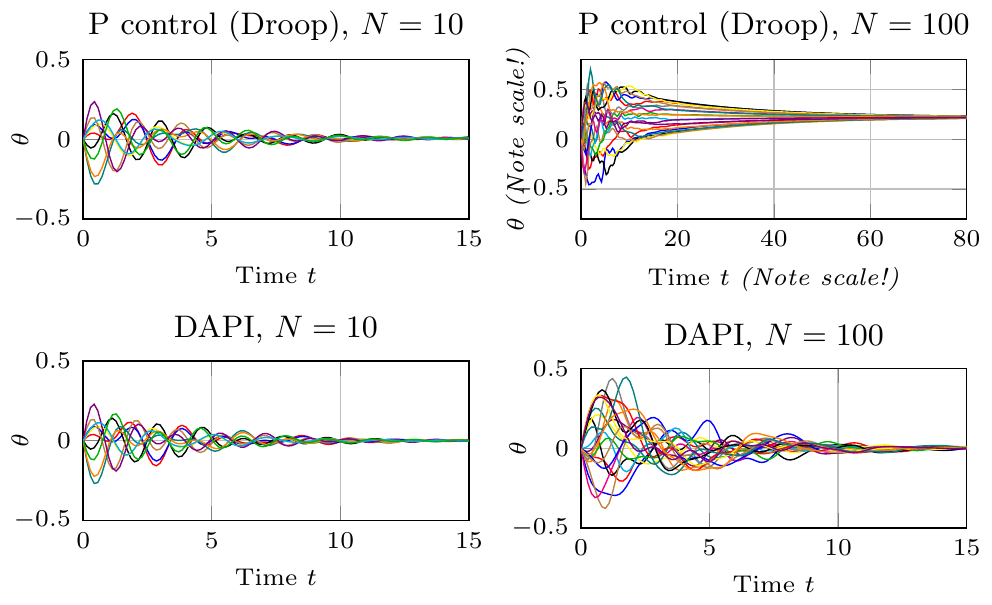}
  \label{fig:dapisim}  } \hspace{0.45cm}
  \subfloat[][ Subset of position trajectories for a 100 vehicle formation, subject to noise. 
  Here, $f = g = f_0 = K_D = 1$, $\tau = 0.1$. ] {
\includegraphics[scale = 1]{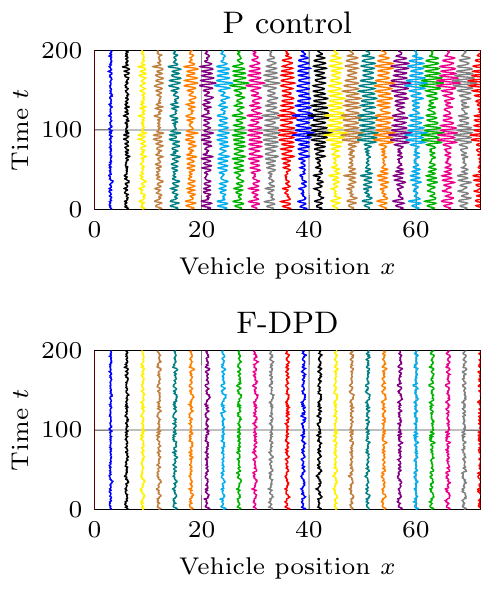}
  \label{fig:fdpdsim}} 
       \caption{Simulations from Examples~1 and 2. }
\label{fig:Sims}
\end{figure*}
%
\subsection{Absolute $x$-feedback: Distributed PD control}
Suppose now that an absolute measurement of $x_i$ is available, while one of $v_i$ is not. In this case, distributed proportional derivative (PD) control can be used to improve performance compared to the P controller. With ideal derivative action, the controller would take the form:
\begin{equation}
\label{eq:idealDPD}
u_i \! = \! -\!\! \!\sum_{j \in \mathcal{N}_i}\!\! f_{ij}(x_i  \! - \! x_j)   -\!\!\! \sum_{j \in \mathcal{N}_i} \! \! g_{ij}(v_i \!- \!v_j)  - \! f_0x_i  - \! K_D \! \frac{\mathrm{d}x_i}{\mathrm{d}t},
\end{equation}
which, since $\frac{\mathrm{d}x_i}{\mathrm{d}t} = v_i(t)$, is identical to~\eqref{eq:static} with absolute feedback from both $x_i$ and $v_i$, substituting the derivative gain $K_D$ for $g_0$. Clearly, this controller would therefore also have the same performance. Ideal derivative action is, however, neither possible nor desirable to implement, partly due to the sensitivity to high frequency noise. We therefore consider a controller where the derivative action is low-pass filtered:
\begin{align} \nonumber
u_i &= - \!\! \!\sum_{j \in \mathcal{N}_i} f_{ij}(x_i - x_j)  - \!\!\!\sum_{j \in \mathcal{N}_i} g_{ij}(v_i - v_j)  - {f_0x_i } + z_i\\
\dot{z}_i &= -\frac{1}{\tau}z_i - \frac{1}{\tau}K_D  {\frac{\mathrm{d}x_i}{\mathrm{d}t}}.
\label{eq:fdpd}
\end{align}
Here, the state $z_i$ corresponds to the derivative action, and $\tau>0$ is the time constant of the filter. We will refer to the controller~\eqref{eq:fdpd} as \textit{filtered distributed PD control (F-DPD)}. Note that the ideal PD controller is retrieved for $\tau = 0$.
\begin{rem}
An alternative approach would be to apply ideal derivative action and low-pass filter the entire control signal $u_i$ in~\eqref{eq:idealDPD}. The resulting performance scaling is similar. We choose the controller~\eqref{eq:fdpd} to enable a better comparison to the P controller~\eqref{eq:static}, as the proportional action is modeled as unfiltered in both.  
\end{rem}

The system~\eqref{eq:doubledyn} with the controller~\eqref{eq:fdpd} becomes:
\begin{equation}
\label{eq:FDPDvec}
\begin{bmatrix}
\dot{x} \\ 
\dot{v} \\ \dot{z}
\end{bmatrix} = \begin{bmatrix}
0 & I & 0 \\ 
- \! \mathcal{L}_F \! - \! f_0I & - \! \mathcal{L}_G & I \\
0 & - \! \frac{1}{\tau} K_D I & - \! \frac{1}{\tau}I 
\end{bmatrix} \!\! \begin{bmatrix}
x\\ v \\ z
\end{bmatrix}  + \begin{bmatrix}
0 \\ I \\ 0
\end{bmatrix} \! w.
\end{equation}
The performance of this system, under Assumption~\ref{ass:proportional}, is given by the following proposition:

\begin{proposition}
\label{thm:fdpd}
The scaled \newtext{performance} output variance~\eqref{eq:pernodevariance} for the F-DPD-controlled system~\eqref{eq:FDPDvec} is given by

\vspace{-3.5mm}
\begin{small}
\begin{equation}
\label{eq:normdfpf}
V_N^{\mathrm{F-DPD}} \!\! = \! \frac{1}{2N} \!\! \sum_{n = 2}^N \! \frac{1}{  (f_0\!+\!f\ldn) \!  \left( \! g\ldn \!+ \!\frac{K_D(\tau g \ldn \!+\!1 )}{\tau^2 (f_0\!+ \! f\ldn)\! +\! \tau g \ldn \!+ \!1} \! \right)   }
\end{equation}
\end{small}
\end{proposition}
\vspace{1.5mm}

\begin{proof}
This result is derived in the same manner as in the proof of Lemma~\ref{thm:static}. \newtext{The system's stability can in this case be verified through the characteristic polynomial $\xi^3 +(g\lambda_n + \frac{1}{\tau}) \xi^2 + (\frac{K_D+g\lambda_n}{\tau} +f_0 +f\lambda_n)\xi +\frac{f_0 +f\lambda_n}{\tau} = 0$, which satisfies the Routh-Hurwitz criterion.}
\end{proof}
Again, the expression is convoluted, but it reveals that the output variance is bounded for any network:
\begin{corollary}
\label{thm:fdpdcor}
For any positive and finite gains $K_D$ and $\tau$, the variance $V_N^{\mathrm{F-DPD}}$ in~\eqref{eq:normdfpf} is uniformly bounded in $N$. It holds that 
\begin{equation}
\label{eq:boundsfdpd}
0 < V_N^{\mathrm{F-DPD}} < \frac{\tau^2f_0 +1}{2f_0K_D}
\end{equation}
\end{corollary} \vspace{2mm}
\begin{proof}
Similar to the proof of Corollary~\ref{thm:dapicor}, it can be verified that $s_n =\frac{1}{  (f_0\!+\!f\ldn) \!  \left( \! g\ldn \!+ \!\frac{K_D(\tau g \ldn \!+\!1 )}{\tau^2 (f_0\!+ \! f\ldn)\! +\! \tau g \ldn \!+ \!1} \! \right)} $ from the sum in~\eqref{eq:normdfpf} is monotonically decreasing in $\ldn>0$. An upper bound is thus obtained as $\ldn \rightarrow 0$, which gives~$s_n < \frac{\tau^2f_0 +1}{f_0K_D} $ for all $n$. Therefore, $V_N^{\mathrm{F-DPD}} = \frac{1}{2N} \sum_{n = 2}^N s_n < \frac{N-1}{N} \frac{\tau^2f_0 +1}{f_0K_D} < \frac{\tau^2f_0 +1}{2f_0K_D}$.
\end{proof}

\vspace{1mm}
\setcounter{example}{0}
\begin{example}[Continued]
Consider again the formation of vehicles. Since no absolute feedback from velocities $v_i$ was available, we implement the F-DPD controller~\eqref{eq:fdpd}. Fig.~\ref{fig:fdpdsim} shows a simulation of a 100 vehicle formation under a white noise disturbance input. While the spacings between neighbors seem well-regulated in both systems, the P-controlled system appears less coherent. 
\end{example}

%

\section{Controller tuning for improved coherence}
\label{sec:tuning}
Neither the DAPI controller~\eqref{eq:dapi}, nor the F-DPD controller~\eqref{eq:fdpd} represents ideal integral or derivative action. In both cases, we have modeled filters to mitigate well-known issues related to noise and uncertainties. In this section, we discuss how the design of those filters impacts performance. 

\subsection{Optimal distributed averaging in DAPI}

As discussed in Section~\ref{sec:dapi}, robust stability of the DAPI controlled system~\eqref{eq:DAPIvec} requires an alignment between the agents' integral states $z_i$ through a distributed averaging filter. While the control objectives will be reached for any non-zero gains $c_{ij} $ \cite{SimpsonPorco2015}, an important control design question is how to choose these gains to optimize network performance. In our case, this choice is reflected through the constant $c$ (see Assumption~\ref{ass:uniformC}). Consider the following proposition. 
\begin{proposition}
\label{thm:cstarlemma}
For a given DAPI controlled system~\eqref{eq:DAPIvec}, there is \emph{either} a $c^\star>0$ that minimizes $V_N^{\mathrm{DAPI}}$ in~\eqref{eq:normdapi}, \emph{or} $V_N^{\mathrm{DAPI}}$ is minimized by $c^\star = 0$ and is monotonically increasing for $c>0$. 
The case $c^\star>0$ holds if $f > \frac{1}{\ldn}(g\ldn + g_0)^2$ for all $n = 2,\ldots,N$ and the case $c^\star = 0$ holds if $f \le \frac{1}{\ldn}(g\ldn + g_0)^2$ for all $n = 2,\ldots,N$. If none of these applies, $c^\star$ must be evaluated case-by-case.
\end{proposition}
\begin{proof}
See Appendix.
\end{proof}
Finding a closed form expression for $c^\star$ is in general not tractable due to the many terms in the sum~\eqref{eq:normdapi}, and must be done on a case-by-case basis. However, it is possible in the special case of a complete network graph: 
\begin{corollary}
If the network graph $\mathcal{G}$ is complete and the edge weights $l_{ij} = l$ for all $\{i,j\} \in \mathcal{E}$, then $c^\star$ is given by 
\begin{equation}
\label{eq:cstarcomplete}
c^\star = \sqrt{\frac{f}{Nl}} - g +\frac{g_0}{Nl}
\end{equation}
if this is a positive number. Otherwise $c^\star = 0$.
\end{corollary}
\begin{proof}
In this case, the eigenvalues $\lambda_n = Nl$ for $n\ge 2$, and~\eqref{eq:cstarcomplete} follows from the proof of Proposition~\ref{thm:cstarlemma}.
\end{proof}
\begin{rem}
An explicit inclusion of measurement noise in the model increases the output variance $V_N^{\mathrm{DAPI}}$, and shifts the optimal $c^\star$. Such a model reveals that $c = 0$ is always suboptimal, since $V_N^\mathrm{DAPI}$ is then infinite.
\end{rem}

It is interesting to note that in cases where $c^\star >0 $, the performance of the DAPI controlled system at the optimum is better than with P control and absolute feedback from both $x_i$ and $v_i$ (as that performance is retrieved for $c=0$, see Remark~\ref{rem:zeroc}). 
Looking at e.g.~\eqref{eq:cstarcomplete}, we note that this occurs in particular if $g$ is  small relative to $f$, so that there is little alignment in the state $v$ between agents. This is the case in power networks, where in general  $g = 0$ (see Example~\ref{ex:power}). 
%
%

\subsection{Impact of low-pass filter in F-DPD}
The low pass filter in the F-DPD controller~\eqref{eq:fdpd} is included as a more realistic implementation of derivative action, which is otherwise well-known to be highly sensitive to high frequency variations and noise. While, by Corollary~\ref{thm:fdpdcor}, F-DPD achieves bounded output variance for any finite filter constant $\tau$, choosing a small value (corresponding to a high bandwidth) gives better performance:
\begin{proposition}
\label{thm:filter}
For a given F-DPD controlled system~\eqref{eq:FDPDvec}, the output variance $V_N^{\mathrm{F-DPD}}$ in~\eqref{eq:normdfpf} is minimized by $\tau = 0$ and is monotonically increasing for $\tau>0$.
\end{proposition}
\vspace{1mm}
\begin{proof}
It holds that $\frac{\mathrm{d} }{\mathrm{d} \tau} V_N^{\mathrm{F-DPD}} = \frac{1}{2N} \sum_{n = 2}^N \frac{K_D g \ldn^2 \tau^2 + 2 K_D \ldn \tau}{2 (g^2 \ldn^2 \tau + f g \ldn^2 \tau^2 + f_0 g \ldn \tau^2 + K_D g \ldn \tau + g \ldn + K_D)^2} $. Since $K_D,\ldn,f_0>0$, $f,g \ge 0$, $\frac{\mathrm{d} }{\mathrm{d} \tau} V_N^{\mathrm{F-DPD}}  = 0$ for $\tau = 0$ and $\frac{\mathrm{d} }{\mathrm{d} \tau} V_N^{\mathrm{F-DPD}}  > 0$ for all $\tau > 0$.
\end{proof}
\vspace{0.6mm}

This is an intuitive result, as $\tau = 0$ would give ideal derivative action and therefore the most accurate substitute for absolute feedback from $v_i$. However, a greater value for $\tau$ would make the system less sensitive to noise. This trade-off must be done based on system-specific knowledge.

\begin{rem}
We have limited the analysis here to a first order filter, as higher order filters make the expressions for output variances even more convoluted. However, numerical results indicate that higher order filters increase $V_N^{\mathrm{F-DPD}}$, though it remains bounded.
\end{rem}

\section{Conclusions}
\label{sec:discussion}
In this paper, we have addressed limitations to the performance of the standard proportional consensus protocol in double-integrator networks. These limitations imply that the variance measure $V_N$ in~\eqref{eq:pernodevariance}, that characterizes lack of network coherence, scales unboundedly with network size in sparse networks, unless absolute feedback from both the system's state is available. We addressed these shortcomings by proposing distributed proportional integral (PI) and proportional derivative (PD) controllers, which radically improve performance \newtext{by making $V_N$ bounded} if absolute measurements from one of the system's states are available. 

The second-order consensus problem considered herein can be used to model a variety of applications ranging from biological networks to coordination of robots, see, for example,~\cite{RenAtkins2005}. In this paper, we have presented two examples from vehicular formations and power networks. Our results show that if agents have limited access to absolute feedback, the proposed PI and PD controllers are preferable to P control, in particular for large-scale networked systems. We remark that this conclusion holds even though the proposed controllers do not model \textit{ideal} integral and derivative action, but have filters that mitigate the effect of noise that would arise in a practical setting. 
\newtext{The impact of other practical network aspects, such as communication delays, non-reliability of the channels and non-symmetries, is an open and important research question that is left for future work. }

%
%
%
%

\section{Acknowledgements}
We would like to thank Bassam Bamieh for many interesting discussions and insightful comments related to this work. We are also grateful to Martin Andreasson and Hendrik Flamme for a number of incisive discussions, \newtext{as well as to anonymous reviewers for their valuable feedback.}

\section*{Appendix}
\subsection*{Proof of Lemma~\ref{thm:static}}
The output variance $V$ in~\eqref{eq:perfmeas} is obtained as the squared \hn norm of the system $\mathcal{S}$ from input $w$ in~\eqref{eq:staticvec} to the output~\eqref{eq:ydef}. \hn norm calculations for the types of systems considered herein have been presented in detail in \cite[Chapter 2]{Tegling2016Lic}, so we only give a summary proof here. We derive the norm through a unitary state transformation: $x =: U\hat{x}$, $v =: U\hat{v}$, where $U$ is the unitary matrix that diagonalizes the Laplacian matrix $\cl$ (and by Assumption~\ref{ass:proportional} also $\cl_F,\cl_G$), so that $\cl = U^*\Lambda U $ with $ \Lambda = \mathrm{diag}\{\lambda_1, \cdots, \lambda_N \}$. Due to the unitary invariance of the \hn norm, we can also transform the input and output according to $\hat{w} := U^*w,~\hat{y} := U^*y$.

This state transformation block-diagonalizes the system~\eqref{eq:staticvec}, which can now be described through $N$ decoupled subsystems $\hat{\mathcal{S}}_n$:
\begin{equation}
\begin{aligned}
\begin{bmatrix}
\dot{\hat{x}}_n \\ 
\dot{\hat{v}}_n
\end{bmatrix}  & = \underbrace{ \begin{bmatrix}
0 & 1 \\ 
- f  \lambda_n   -   f_0 & -  g  \ldn   -g_0
\end{bmatrix}}_{=:\hat{A}_n}    \begin{bmatrix}
\hat{x}_n\\ \hat{v}_n
\end{bmatrix}  +  \underbrace{ \begin{bmatrix}
0 \\ 1
\end{bmatrix}}_{=:\hat{B}_n}  \hat{w}_n   \\ 
\hat{y}_n &  =  \underbrace{\begin{bmatrix}
1 & 0 
\end{bmatrix}}_{=:\hat{C}_n} \begin{bmatrix}
\hat{x}_n\\ \hat{v}_n
\end{bmatrix}
\end{aligned}
\end{equation}
for $n = 2, \ldots,N $. 
The mode $n = 1$ corresponds to the average mode in $x$, which is unobservable from the output~\eqref{eq:ydef}, 
so $\hat{y}_1 \equiv 0$. \newtext{Therefore, even though this mode would be undamped if $f_0 = 0$, the overall system remains input-output stable since $\hat{A}_n$ is Hurwitz for $n\ge 2$ (its characteristic polynomial is $\xi^2 + (g_0 + g\lambda_n)\xi + f_0 + f \lambda_n =0$, which has only positive coefficients).   }

It holds that $V = ||\mathcal{S}||_2^2 = \sum_{n = 1}^N ||\hat{\mathcal{S}}_n||_2^2 =  \sum_{n = 2}^N ||\hat{\mathcal{S}}_n||_2^2 $, where the last equality is due to the fact that $\hat{y}_1 \equiv 0$ and therefore $||\hat{\mathcal{S}}_1||_2^2 = 0$. Each subsystem norm is obtained as $||\hat{\mathcal{S}}_n||_2^2 = \mathrm{tr}\{\hat{B}_n P_n\hat{B}_n\}$, where $P_n$ is the solution to the Lyapunov equation $\hat{A}_n^*P_n + P_n\hat{A}_n = -\hat{C}_n^*\hat{C}_n$.  Straightforward calculations give that $||\hat{\mathcal{S}}_n||_2^2 = \frac{1}{2(f\ldn + f_0)(g\ldn + g_0)}$ for each~$n$ and the result~\eqref{eq:normstatic} follows. 

\subsection*{Proof of Proposition~\ref{thm:cstarlemma}}
Clearly, $\frac{\mathrm{d}}{\mathrm{d}c}V_N^{\mathrm{DAPI}} =  \frac{1}{2N}\sum_{n = 2}^N \frac{\mathrm{d}}{\mathrm{d}c} s_n $, where $s_n$ was defined in the proof of Corollary~\ref{thm:dapicor}. Each term $s_n$ can be written as a fraction $s_n=\frac{p_n}{q_n}$. By the quotient rule,  $\frac{\mathrm{d}s_n}{\mathrm{d}c} = \frac{p'_nq_n - p_nq'_n}{(q_n)^2}$, where~$(q_n)^2 >0$ for all~$c$, and~$p'_nq_n - p_nq'_n =K_I\ldn^2 [  c^2 + 2( g + \frac{g_0}{\ldn} )c + ( g + \frac{g_0}{\ldn} )^2 - \frac{f}{\ldn} ] $.  Now, since $K_I,\ldn,g_0>0$, and $f,g \ge 0$ $,\frac{\mathrm{d}s_n}{\mathrm{d}c} >0$ for all $c>0$ if $\frac{f}{\ldn} \le ( g + \frac{g_0}{\ldn} )^2$ for all $n = 2,\ldots,N$. In this case $\frac{\mathrm{d}}{\mathrm{d}c}V_N^{\mathrm{DAPI}}>0$ for all $c>0$, and $c^\star = 0$. Conversely, if $\frac{f}{\ldn} > ( g + \frac{g_0}{\ldn} )^2$ for all $n = 2,\ldots,N$ then $\frac{\mathrm{d}s_n}{\mathrm{d}c} \big|_{c = 0} <0$, and each $s_n$ is minimized by \newtext{the positive root of the quadratic function above, and} $c^\star_n>0$.  Therefore, $\frac{\mathrm{d}}{\mathrm{d}c}V_N^{\mathrm{DAPI}}\big|_{c = 0} <0 $ and some $c^\star >0$ minimizes $V_N^{\mathrm{DAPI}}$. If $\frac{f}{\ldn} > ( g + \frac{g_0}{\ldn} )^2$ for some, but not all $n$, the existence of a positive minimizer $c^\star >0$ depends on remaining system parameters.

\bibliographystyle{IEEETran}
\bibliography{EmmasBib17_NonPS}

\end{document}